\documentclass[a4paper,12pt]{article}

\usepackage{amsfonts}
\usepackage{graphicx}
\usepackage{subfig}
\usepackage{amsmath}
\usepackage{amssymb}
\usepackage{enumitem}
\usepackage{url}
\usepackage{color}
\usepackage{authblk}
\usepackage{amsthm}
\usepackage{multicol}

\theoremstyle{definition}

\theoremstyle{theorem}
\newtheorem{theorem}{Theorem}[section]
\theoremstyle{corollary}
\newtheorem{corollary}{Corollary}[section]
\theoremstyle{proposition}
\newtheorem{proposition}{Proposition}[section]
\theoremstyle{remark}
\newtheorem{remark}{Remark}

\title{\bf {\sc Explicit solutions for a non-classical heat conduction problem for a semi-infinite strip with a non-uniform heat source}}\vspace{1cm}

\author[$\dag$]{Andrea N. Ceretani}
\author[$\dag$]{Domingo A. Tarzia}
\author[$\ddag$]{Luis T. Villa}

\affil[$\dag$]{\small {CONICET-Depto. Matem\'atica, Facultad de Ciencias Empresariales, Universidad Austral,
Paraguay 1950, S2000FZF Rosario, Argentina. E-mail: \textcolor{blue}{aceretani@austral.edu.ar}; \textcolor{blue}{dtarzia@austral.edu.ar}}}

\affil[$\ddag$]{\small INIQUI (CONICET - UNSa), Facultad de Ingenier\'ia, Universidad Nacional de Salta, Av. Bolivia 5150, 4400 Salta, Argentina. E-mail: \textcolor{blue}{villal@unsa.edu.ar}}

\date{}

\begin{document}
\maketitle

\thispagestyle{empty} %Si es que no se desea numerar alguna p\'agina en particular (por ej. la que contiene al t\'itulo).

\begin{abstract}
A non-classical initial and boundary value problem for a non-homogeneous one- dimensional heat equation for a semi-infinite material $x>0$ with a zero temperature boundary condition at the face $x=0$ is studied with the aim of finding explicit solutions. It is not a standard heat conduction problem because a heat source $-\Phi(x)F(V(t),t)$ is considered, where $\Phi$ and $F$ are real functions and $V$ represents the heat flux at the face $x=0$.

Explicit solutions independents of the space or temporal variables are given. Solutions with separated variables when the data functions are defined from the solution $X=X(x)$ of a linear initial value problem of second order and the solution $T=T(t)$ of a non-linear (in general) initial value problem of first order which involves the function $F$, are also given and explicit solutions corresponding to different definitions of the function $F$ are obtained. A solution by an integral representation depending on the heat flux at the boundary $x=0$ for the case in which
$F=F(V,t)=\nu V$, for some $\nu>0$, is obtained and explicit expressions for the heat flux at the boundary $x=0$ and for its corresponding solution are calculated when $h=h(x)$ is a potential function and $\Phi=\Phi(x)$ is given by $\Phi(x)=\lambda x$, $\Phi(x)=-\mu\sinh{(\lambda x)}$ or $\Phi(x)=-\mu\sin{(\lambda x)}$, for some $\lambda>0$ and $\mu>0$.

The limit when the temporal variable $t$ tends to $+\infty$ of each explicit solution obtained in this paper is studied and the "controlling" effects of the source term $-\Phi F$ are analysed by comparing the asymptotic behaviour of each solution with the asymptotic behaviour of the solution to the same problem but in absence of source term.

Finally, a relationship between this problem with another non-classical initial and boundary value problem for the heat equation is established and explicit solutions for this second problem are also obtained.

As a consequence of our study, several problems which can be used as benchmark problems for testing new numerical methods for solving partial differential equations are obtained.
\end{abstract}


{\bf 2010 AMS subjet classification}: 35C05, 35C15, 35C20, 35K55, 45D05, 80A20.

\newpage

\section{Introduction}

In this paper we study the following non-classical initial and boundary value problem for a non-homogeneous one-dimensional heat equation (Problem P):
\begin{align}%\label{pb:P}
\label{eq:u}&u_t(x,t)-u_{xx}(x,t)=-\Phi(x)F(u_x(0,t),t)\hspace{0.5cm}	&x>0,\hspace{0.5cm}&t>0 \\
\label{ic:u}&u(x,0)=h(x)\hspace{0.5cm}	&x>0\hspace{0.5cm}&\\
\label{bc:u}&u(0,t)=0	&	&t>0
\end{align}

with the aim of finding explicit solutions, where $u=u(x,t)$ is the unknown temperature function, defined for $x\geq0$ and $t\geq0$, $\Phi=\Phi(x)$, $h=h(x)$ and $F=F(V,t)$ are given functions defined, respectively, for $x>0$ and $V\in\mathbb{R}$, $t>0$, and the function $h$ satisfies the following compatibility condition:
\begin{equation}\label{cc:h}
\displaystyle\lim_{x\to 0^+}h(x)=0\text{.}
\end{equation}

This problem is motivated by the regulation of the temperature $u=u(x,t)$ of an isotropic medium, occupying the semi-infinite spatial region $x>0$, under a non-uniform heat source $-\Phi(x)F(u_x(0,t),t)$ which provides a heater or cooler effect depending on the properties of the function $F$ respect to the heat flux $u_x(0,t)$ at the boundary $x=0$ \cite{Ca1984,CaJa1959}. For example, if:
\begin{equation*}
\Phi(x)>0\hspace{0.5cm}\text{ for }x>0,\hspace{0.5cm}\text{ and }\hspace{0.5cm}u_x(0,t)F(u_x(0,t),t)>0\hspace{0.5cm}\forall\,t>0\text{,}
\end{equation*}
then the source term is a cooler when $u_x(0,t)>0$ and a heater when $u_x(0,t)<0$. Some references in this subjet are \cite{BeTaVi2000,CaYi1989,Fr1964,GlSp1981,GlSp1982,Ke1990,KePr1998,LiMu2014,
QuLo2014,So1998,TaVi1998,Vi1986}. Problem P for the slab $0<x<1$ has been studied in \cite{SaTaVi2011}. Recently, free boundary problems (Stefan problems) for the non-classical heat equation have been studied in \cite{BrNa2014-2,BrTa2006-1,BrTa2006-2,BrTa2010,DuLo2014,KoVa2010}, where some explicit solutions are also given, and a first study of non-classical heat conduction problem for a $n$-dimensional material has been given in \cite{BoTa2014}. Numerical schemes for Problem P when a non-homogeneous boundary condition is considered have been studied in \cite{Ko2008} and numerical solutions have been given for two particular choices of data function corresponding to problems with known explicit solutions.

The organization of the paper is the following: in Section \ref{s:EsP}, we give explicit solutions to Problem P. We split this section into three parts. In the first one, we give explicit solutions which are independents of the space variable $x$ or the temporal variable $t$. In the second part, we find solutions with separated variables when the functions $h=h(x)$ and $\Phi=\Phi(x)$ are proportional to the solution $X=X(x)$ of a linear initial value problem of second order and the function $F=F(V,t)$ is defined from the solution $T=T(t)$ of a non-linear (in general) initial value problem of first order. As a consequence, we give explicit solutions with separated variables corresponding to different definitions of the function $F$. Finally, in the third part, we find solutions by an integral representation which depends on the heat flux at the boundary $x=0$ \cite{TaVi1998} for the case in which $F$ is defined by $F(V,t)=\nu V$, for some $\nu>0$. Moreover, we find explicit expressions for the heat flux at the boundary $x=0$ and for its corresponding solution to Problem P, when $h=h(x)$ is a potential function and $\Phi=\Phi(x)$ is given by $\Phi(x)=\lambda x$, $\Phi(x)=-\mu\sinh{(\lambda x)}$ or $\Phi(x)=-\mu\sin{(\lambda x)}$, for some $\lambda>0$ and $\mu>0$. In Section \ref{s:Cp}, we deal with the problem of "controlling" solutions of Problem P through the source term $-\Phi(x)F(u_x(0,t),t)$. We compare the asymptotic behaviour of each explicit solution $u$ obtained for the Problem P with the asymptotic behaviour of the solution $u_0$ of the same problem but in absence of source term, and we obtain conditions for the parameters involved in the definition of $-\Phi(x)F(V,t)$ under which the asymptotic behaviour of $u$ can be controlled respect to the asymptotic behaviour of $u_0$. Finally, in Section \ref{s:EsPP}, we recall the relationship between Problem P with another non-classical initial and boundary value problem for the heat equation \cite{TaVi1998}, given by (Problem $\widetilde{P}$):
\begin{align}%\label{pb:PP}
\label{eq:v}&v_t(x,t)-v_{xx}(x,t)=-\widetilde{\Phi}(x)\widetilde{F}(v(0,t),t)\hspace{0.5cm}	&x>0,\hspace{0.5cm}&t>0 \\
\label{ic:v}&v(x,0)=\tilde{h}(x)\hspace{0.5cm}	&x>0\hspace{0.5cm}&\\
\label{bc:v}&v_x(0,t)=\tilde{g}(t)	&	&t>0
\end{align}
and we find explicit solutions to Problem $\widetilde{P}$.

As a consequence of our study, we obtain some particular cases of Problem P and Problem $\widetilde{\text{P}}$ which can be used as benchmark problems for testing new numerical methods for solving partial differential equations.

\section{Explicit solutions for Problem P}\label{s:EsP}

\subsection{Explicit solutions independents of space or temporal variables}\label{ss:In}

\begin{theorem}\label{th:In}
\begin{enumerate}
\item[]
\item[1.] Problem P does not admit any non-trivial solution independent of the space variable $x$.
\item[2.] If:
\begin{enumerate}
\item $F$ is the zero function and $h$ is defined by:
\begin{equation}\label{th_In:h}
h(x)=\eta x,\hspace{0.5cm}x\geq0\text{,}
\end{equation}
for some $\eta>0$,
\end{enumerate}
or
\begin{enumerate}[resume]
\item $F$ is a constant function defined by:
\begin{equation}\label{th_In:F}
F(V,t)=\nu,\hspace{0.5cm}V\in\mathbb{R},\,t>0\text{,}
\end{equation}
for some $\nu\in\mathbb{R}-\{0\}$, and $h$ is a twice differentiable function  such that $h(0)$ exists and:
\begin{equation}\label{th_In:h''}
h''(x)=\nu\Phi(x),\hspace{0.5cm}x>0\text{,}
\end{equation}
\end{enumerate}
then the function $u$ defined by:
\begin{equation}\label{th_In:uu}
u(x,t)=h(x),\hspace{0.5cm}x\geq0,\,t\geq0\text{,}
\end{equation}
is a solution to Problem P independent of the temporal variable $t$.
\end{enumerate}
\end{theorem}

\begin{proof}
\begin{enumerate}
\item[]
\item[1.] If the Problem P has a solution $u$ independent of the space variable $x$ then:
\begin{equation}
u(x,t)=u(0,t)=0,\,x>0,\,t>0\hspace{0.5cm}\text{and}\hspace{0.5cm}u(0,0)=\displaystyle\lim_{x\to 0^+}h(x)=0\text{.}
\end{equation}
Therefore $u$ is the zero function.
\item[2.] It is easy to check that the function $u$ given in (\ref{th_In:uu}) is a solution to the Problem P given in this item.
\end{enumerate}
\end{proof}

\subsection{Explicit solutions with separated variables}\label{ss:Sv}

\begin{theorem}\label{th:Sv}
Let $\lambda,\,\eta,\,\delta\in\mathbb{R}-\{0\}$. If $\Phi$, $h$ and $F$ are defined by:
\begin{equation}\label{th_Sv:h-Phi-F}
\Phi(x)=\lambda X(x),\,\,h(x)=\eta X(x),\,x>0\hspace{0.25cm}\text{and}\hspace{0.25cm}F=F(\delta T(t),t),\,t>0\text{,}
\end{equation}
where $X$ is given by:
\begin{equation}\label{th_Sv:X}
X(x)=\left\{\begin{array}{ccc}
             \frac{\delta}{\sqrt{\sigma}}\sinh{(\sqrt{\sigma}x)} & &\text{if }\sigma>0\\
             \frac{\delta}{\sqrt{|\sigma|}}\sin{(\sqrt{|\sigma|}x)}& & \text{if }\sigma<0\\
			 \delta x& &\text{if }\sigma=0
            \end{array}\right.,\hspace{0.5cm}x\geq0
\end{equation}
and $T$ is the solution of the initial value problem:
\begin{align}
\label{th_Sv:T}&\dot{T}(t)-\sigma T(t)=-\lambda F(\delta T(t),t),\hspace{0.5cm}t>0\\
\label{th_Sv:TT}&T(0)=\eta \text{,}
\end{align}
then the function $u$ given by:
\begin{equation}\label{th_Sv:u}
u(x,t)=X(x)T(t),\hspace{0.5cm}x\geq 0,\,t\geq 0
\end{equation}
is a solution with separated variables to Problem P.
\end{theorem}

\begin{proof}
An easy computation shows that the function $u$ given in (\ref{th_Sv:u}) is a solution to Problem P.
%It follows from elementary calculations for Problem P.
\end{proof}

\begin{remark}
The function $X$ given in (\ref{th_Sv:X}) also can be seen as the solution  of a linear initial value problem of second order, in fact $X$ satisfies:
\begin{align}
&X''(x)-\sigma X(x)=0,\hspace{0.5cm}x>0\\
&X(0)=0\\
&X'(0)=\delta\text{.}
\end{align}
\end{remark}

Under the hypothesis of the previous theorem, the problem of finding explicit solutions with separated variables to Problem P reduces to solving the initial value problem (\ref{th_Sv:T})-(\ref{th_Sv:TT}).

With the spirit of exhibit explicit solutions to Problem P, our next result summarizes explicit solutions to the initial value problem (\ref{th_Sv:T})-(\ref{th_Sv:TT}) corresponding to three different definitions of the function $F$.

\begin{proposition}\label{pr:Sv}
If in Theorem \ref{th:Sv} we consider:
\begin{enumerate}
\item Function $F$ defined by:
\begin{equation}\label{pr_Sv:F}
F(V,t)=\nu V,\hspace{0.5cm}V\in\mathbb{R},\,t>0\text{,}
\end{equation}
for some $\nu\in\mathbb{R}-\{0\}$, then the function $T$ is given by:
\begin{equation}\label{pr_Sv:T}
T(t)=\eta\exp{((\sigma-\lambda\nu\delta) t)},\hspace{0.5cm}t\geq 0\text{.}
\end{equation}
\item Function $F$ defined by:
\begin{equation}\label{pr_Sv:FF}
F(V,t)=f_1(t)+f_2(t)V,\hspace{0.5cm}V\in\mathbb{R},\,t>0\text{,}
\end{equation}
for some $f_1, f_2\in L^1_{loc}(\mathbb{R}^+)$,
then the function $T$ is given by:
\begin{equation}\label{pr_Sv:TT}
T(t)=g_1(t)\exp{(g_2(t))},\hspace{0.5cm}t\geq 0\text{,}
\end{equation}
where functions $g_1$ and $g_2$ are defined by:
\begin{equation}
g_1(t)=\eta-\lambda\displaystyle\int_0^tf_1(\tau)\exp{\left(\lambda\delta\displaystyle\int_0^\tau f_2(\xi)d\xi-\sigma \tau\right)}
d\tau,\hspace{0.5cm}t\geq 0
\end{equation}
\begin{equation}
g_2(t)=\sigma t-\lambda\delta\displaystyle\int_0^tf_2(\tau)d\tau,\hspace{0.5cm}t\geq 0\text{.}
\end{equation}
\item Function $F$ defined by:
\begin{equation}\label{pr_Sv:FFF}
F(V,t)=V^nf(t),\hspace{0.5cm}V\in\mathbb{R},\,t>0\text{,}
\end{equation}
for some $n<1$ and some positive function $f\in L^1_{\text{loc}}(\mathbb{R}^+)$, and $\lambda$, $\delta$ and $\eta$ positive numbers, then the function $T$ is given by:
\begin{equation}\label{pr_Sv:TTT}
T(t)=g(t)\exp{(\sigma t)},\hspace{0.5cm}t\geq 0\text{,}
\end{equation}
where the function $g$ is defined by:
\begin{equation}
g(t)=\left(\eta^{1-n}+\lambda\delta^n(n-1)\displaystyle\int_0^tf(\tau)\exp{(\sigma(n-1) \tau)}d\tau\right)^{\frac{1}{1-n}},\hspace{0.5cm}t\geq 0\text{.}
\end{equation}
\end{enumerate}
\end{proposition}

\begin{proof}
It follows by the application of the integrating factor method to the initial value problem (\ref{th_Sv:T})-(\ref{th_Sv:TT}).
\end{proof}

\subsection{Explicit solutions obtained from an integral representation}\label{ss:Ir}

Our next theorem is a restatement of Theorem 1 in \cite{TaVi1998} for a particular choice of the function $F$ in Problem P.

\begin{theorem}\label{th:Ir}
Let:
\begin{enumerate}
\item $h$ a continuously differentiable function in $\mathbb{R}^+$ such that $h(0)$ exists and there exist positive numbers $\epsilon$, $c_0$ and $c_1$ such that:
\begin{equation}\label{th_Ir:h}
|h(x)|\leq c_0\exp{(c_1x^{2-\epsilon})},\hspace{0.5cm}\forall\,x>0\text{,}
\end{equation}
\item $\Phi$ a locally H$\ddot{o}$lder continuous function
\end{enumerate}
and
\begin{enumerate}[resume]
\item $F$ the function defined by:
\begin{equation}\label{th_Ir:F}
F=F(V,t)=\nu V,\hspace{0.5cm} V\in\mathbb{R},\,t>0\text{,}
\end{equation}
for some $\nu>0$.
\end{enumerate}
If there exists a negative monotone decreasing function $f=f(t)$, defined for $t>0$, such that:
\begin{equation}\label{th_Ir:f}
\displaystyle\int_{t_1}^{t_2}R(t_2-\tau)d\tau\geq f(t_2-t_1),\hspace{0.5cm}\forall\,0<t_1<t_2\text{,}
\end{equation}
where $R$ is defined in function of $\Phi$ by (\ref{th_Ir:R}) (see below), and
\begin{equation}
\displaystyle\lim_{t\to 0^+}f(t)=0\text{,}
\end{equation}
then the function $u$ defined by:
\begin{multline}\label{th_Ir:u}
u(x,t)=\displaystyle\int_0^{+\infty}G(x,t,\xi,0)h(\xi)d\xi-\nu\displaystyle\int_0^t\left(\displaystyle\int_0^{+\infty}G(x,t,\xi,\tau)\Phi(\xi)d\xi\right)V(\tau)d\tau,\\x\geq 0,\,t\geq 0
\end{multline}
is a solution to Problem P, where $G$ is the Green function:
\begin{equation}\label{th_Ir:G}
G(x,t,\xi,\tau)=K(x,t,\xi,\tau)-K(-x,t,\xi,\tau),\hspace{0.5cm} 0<x,\,0<\xi,\,0<\tau<t\text{,}
\end{equation}
being $K$ the fundamental solution of the one-dimensional heat equation:
\begin{equation}\label{th_Ir:K}
K(x,t,\xi,\tau)=\frac{1}{2\sqrt{\pi (t-\tau)}}\exp{(-(x-\xi)^2/4(t-\tau))},\hspace{0.5cm} 0<x,\,0<\xi,\,0<\tau<t\text{,}
\end{equation}
and the function $V$, defined by:
\begin{equation}\label{th_Ir:V}
V(t)=u_x(0,t),\quad t>0\text{,}
\end{equation}
satisfies the Volterra integral equation:
\begin{equation}\label{eq:V}
V(t)=V_0(t)-\nu\displaystyle\int_0^tR(t-\tau)V(\tau)d\tau,\hspace{0.5cm} t>0\text{,}
\end{equation}
where
\begin{equation}\label{th_Ir:V0}
V_0(t)=\frac{1}{\sqrt{\pi t}}\displaystyle\int_0^{+\infty}\exp{(-\xi^2/4t)}h'(\xi)d\xi,\hspace{0.5cm} t>0\text{,}
\end{equation}
and
\begin{equation}\label{th_Ir:R}
R(t)=\frac{1}{2\sqrt{\pi}t^{3/2}}\displaystyle\int_0^{+\infty}\xi\exp{(-\xi^2/4t)}\Phi(\xi)d\xi,\hspace{0.5cm} t>0\text{.}
\end{equation}
\end{theorem}

\begin{remark}
The interest of the previous theorem is that it enable us to finding an explicit solution $u=u(x,t)$ to Problem P by finding the corresponding heat flux $u_x(0,t)$ at the boundary $x=0$ as a solution of the integral equation (\ref{eq:V}).
\end{remark}

The remainder of this section will be devoted to the study of Problem P when:
\begin{enumerate}
\item $F$ is given as in (\ref{th_Ir:F}),
\item $h$ is defined by:
\begin{equation}\label{def:h}
h(x)=\eta x^m,\hspace{0.5cm}x>0\text{,}
\end{equation}
for some $\eta\in\mathbb{R}-\{0\}$ and $m\geq 1$,
\end{enumerate}
and
\begin{enumerate}[resume]
\item $\Phi$ is given by one of the following expressions:
\begin{equation}\label{def:phi123}
\varphi_1(x)=\lambda x,\hspace{0.5cm}\varphi_2(x)=-\mu\sinh{(\lambda x)}\hspace{0.5cm}\text{or}\hspace{0.5cm}\varphi_3(x)=-\mu\sin{(\lambda x)},\hspace{0.5cm}x>0\text{,}
\end{equation}
for some $\lambda>0$ and $\mu>0$.
\end{enumerate}

It is easy to check that for this choice of functions $F$, $h$ and $\Phi$, Problem P is under the hypothesis of the previous theorem (see Appendix \ref{app:1}). Therefore, it has the solution $u=u(x,t)$ given in (\ref{th_Ir:u}).

\begin{proposition}\label{pr:IrV1}
If $F$, $h$ and $\Phi=\varphi_1$ are defined as in (\ref{th_Ir:F}), (\ref{def:h}) and (\ref{def:phi123}), then the heat flux at the boundary $x=0$ corresponding to the solution $u$ (see (\ref{th_Ir:u})) to Problem P is given by:
\begin{equation}\label{pr_IrV1:ux}
u_x(0,t)=\left\{\begin{array}{ccc}
                 \eta\exp{(-\nu\lambda t)}& &\text{if }m=1\\
                 \frac{c(m-1)}{2}\exp{(-\nu\lambda t)}\displaystyle\int_0^t\tau^{(m-3)/2}\exp{(\nu\lambda\tau)}d\tau& &\text{if }m>1
                \end{array}\right.,\hspace{0.5cm}t>0\text{.}
\end{equation}
where
\begin{equation}\label{pr_IrV1:c}
c=\frac{2^{m-1}m\eta}{\sqrt{\pi}}\Gamma\left(\frac{m}{2}\right)\text{,}
\end{equation}
and $\Gamma$ is the Gamma function, defined by:
\begin{equation}
\Gamma(z)=\displaystyle\int_0^{+\infty}\xi^{z-1}\exp{(-\xi)}d\xi,\hspace{0.5cm}z\in\mathbb{R}\text{.}
\end{equation}
\end{proposition}

\begin{proof}
We know from Theorem \ref{th:Ir} that $u_x(0,t)=V(t)$ satisfies the Volterra integral equation (\ref{eq:V}), where the function $V_0$ is given by:
\begin{equation}
V_0(t)=ct^{(m-1)/2},\hspace{0.5cm}t>0\text{.}
\end{equation}
Then, $V(t)$ is given by (see \cite{Mi1971}):
\begin{equation}\label{pr_IrV1:V}
V(t)=\frac{c(m-1)}{2}\displaystyle\int_0^t\tau^{(m-3)/2}r(\tau)d\tau,\hspace{0.5cm}t>0\text{,}
\end{equation}
where $r$ satisfies the integral equation:
\begin{equation}\label{pr_IrV1:r}
r(t)=1-\nu\lambda\displaystyle\int_0^tr(\tau)d\tau,\hspace{0.5cm}t>0\text{,}
\end{equation}
whose solution is given by:
\begin{equation}\label{pr_IrV1:rr}
r(t)=\exp(-\nu\lambda t),\hspace{0.5cm}t>0\text{.}
\end{equation}
By replacing (\ref{pr_IrV1:rr}) in (\ref{pr_IrV1:V}), we obtain (\ref{pr_IrV1:ux}).
\end{proof}

\begin{corollary}\label{co:IrV1}
If in Proposition \ref{pr:IrV1} we consider $m$ an odd number given by $m=2p+1$ with $p\in\mathbb{N}$, then we have:
\begin{equation}\label{co_IrV1:ux}
u_x(0,t)=p_{1,m}(t)-c_{1,m}\exp{(-\nu\lambda t)},\hspace{0.5cm}t>0\text{,}
\end{equation}
where $c_{1,m}$ is given by:
\begin{equation}\label{co_IrV1:c_1m}
c_{1,m}=(-1)^{p-1}\frac{cp!}{(\nu\lambda)^p}\text{,}
\end{equation}
being $c$ the constant given in (\ref{pr_IrV1:c}), and $p_{1,m}(x)$ is the polynomial defined by:
\begin{equation}\label{co_IrV1:p_1m}
p_{1,m}(t)=\left\{\begin{array}{ccc}
                 c_{1,3}& &\text{if }m=3\\
                 -c_{1,5}(\nu\lambda t-1)& &\text{if }m=5\\
				 c_{1,m}\left(\displaystyle\sum_{k=1}^{p-1}\frac{(-\nu\lambda)^k}{k!}t^k+(-1)^{p-1}\right)& &\text{if }m\geq 7
                \end{array}\right.,\hspace{0.5cm}t>0\text{.}
\end{equation}
\end{corollary}

\begin{proof}
It follows by solving the integral in the expression of $u_x(0,t)$ given in (\ref{pr_IrV1:ux}). We do not reproduce these calculations here, but only remark the utility of the identity:
\begin{multline}\label{co_IrV1:int}
\displaystyle\int_0^t\tau^n\exp{(a\tau)}d\tau=
\frac{n!}{a}\exp{(at)}\left(\displaystyle\sum_{k=0}^{n-1}\frac{(-1)^kt^{n-k}}{(n-k)!a^k}+\frac{(-1)^n}{a^n}+\frac{(-1)^{n+1}}{a^n}\exp{(-at)}\right),\\t>0,\,n\in\mathbb{N},\,n\geq 2,\,a\in\mathbb{R}
\end{multline}
when $m\geq 7$.
\end{proof}

Last corollary enables us to obtain the asymptotic behaviour of the heat flux $u_x(0,t)$ at the face $x=0$ when $t$ tends to $+\infty$, for an odd number $m$. Next result is related to this topic. We do not reproduce here the computations involved in its proof, which follows by taking the limit when $t$ tends to $+\infty$ in the expression of $u_x(0,t)$ given in Corollary \ref{co:IrV1}.

\begin{corollary}\label{co:IrVinf1}
If $F$, $h$ and $\Phi=\varphi_1$ are defined as in (\ref{th_Ir:F}), (\ref{def:h}) and (\ref{def:phi123}), where $m$ is an odd number, and $u$ is the solution to Problem P, given in (\ref{th_Ir:u}), then:
\begin{enumerate}
\item if $m=1$, we have:
\begin{equation}
\displaystyle\lim_{t\to+\infty}u_x(0,t)=0\text{,}
\end{equation}
\item if $m=3$, we have:
\begin{equation}
\displaystyle\lim_{t\to+\infty}u_x(0,t)=\frac{6\eta}{\nu\lambda}\text{,}
\end{equation}
\item if $m\geq 5$, we have:
\begin{equation}
\displaystyle\lim_{t\to+\infty}u_x(0,t)=\left\{\begin{array}{ccc}
                                            -\infty& &\text{if }\eta<0\\
                                            +\infty& &\text{if }\eta>0
                                               \end{array}\right.\text{.}
\end{equation}
\end{enumerate}
\end{corollary}

The main idea in the proof of Proposition \ref{pr:IrV1} was to find a solution for the integral equation (\ref{eq:V}) by finding a solution of another integral equation, which was easier to solve. In a more general way, we know that if $V$ satisfies the Volterra integral equation (\ref{eq:V}), with $V_0$ an infinitely differentiable function, then $V(t)$ can be written as (see \cite{Mi1971}):
\begin{equation}\label{re_Ir:V}
V(t)=V_0(0)r(t)+\displaystyle\int_0^tV_0'(t-\tau)r(\tau)d\tau,\hspace{0.5cm}t>0\text{,}
\end{equation}
where $r$ satisfies the integral equation:
\begin{equation}\label{re_Ir:r}
r(t)=1-\nu\displaystyle\int_0^tR(t-\tau)r(\tau)d\tau,\hspace{0.5cm}t>0\text{,}
\end{equation}
and $R$ is given in (\ref{th_Ir:R}). But this last integral equation is not always easy to solve. Nevertheless, in several cases we can find an explicit solution for the equation (\ref{re_Ir:r}) by a formal application of the Laplace transform to their both sides. This is the way which led us to the expressions of $u_x(0,t)$ when $\Phi=\varphi_2$ or $\Phi=\varphi_3$, given in Propositions \ref{pr:IrV2} and \ref{pr:IrV3}.

\begin{proposition}\label{pr:IrV2}
Let $F$, $h$ and $\Phi=\varphi_2$ defined as in (\ref{th_Ir:F}), (\ref{def:h}) and (\ref{def:phi123}), and $\sigma=\lambda+\nu\mu$. Then the heat flux at the boundary $x=0$ corresponding to the solution $u$ (see (\ref{th_Ir:u})) of the Problem P is given by:
\begin{enumerate}
\item If $\sigma\neq 0$ then:
{\small \begin{equation}\label{pr_IrV2:ux}
u_x(0,t)=\left\{\begin{array}{ccc}
                 \frac{\eta}{\sigma}\left(\lambda+\nu\mu\exp{(\lambda\sigma t)}\right)& &\text{if }m=1\\
                 \frac{c\lambda}{\sigma}t^{(m-1)/2}+\frac{c(m-1)\nu\mu}{2\sigma}\exp{(\lambda\sigma t)}\displaystyle\int_0^t\tau^{(m-3)/2}\exp{(-\lambda\sigma\tau)}d\tau& &\text{if }m>1
                \end{array}\right.,\hspace{0.25cm}t>0\text{.}
\end{equation}}
\item If $\sigma=0$ then:
\begin{equation}\label{pr_IrV2:Ux}
u_x(0,t)=\left\{\begin{array}{ccc}
                 \eta(1-\lambda^2t)& &\text{if }m=1\\
                 c t^{(m-1)/2}-\frac{2c\lambda^2}{m+1}t^{(m+1)/2}& &\text{if }m>1
                \end{array}\right.,\hspace{0.5cm}t>0\text{.}
\end{equation}
\end{enumerate}
\end{proposition}

\begin{proof}
An easy computation shows that the expressions given in (\ref{pr_IrV2:ux}) and (\ref{pr_IrV2:Ux}) satisfy the integral equation (\ref{eq:V}). Therefore, they correspond to the heat flux $u_x(0,t)$ at the boundary $x=0$ for the solution $u$ of the Problem P given in (\ref{th_Ir:u}).
\end{proof}

\begin{corollary}\label{co:IrV2}
If in Proposition \ref{pr:IrV2} we consider $\sigma\neq 0$ and $m$ an odd number given by $m=2p+1$ with $p\in\mathbb{N}$, then we have:
\begin{equation}\label{co_IrV2:ux}
u_x(0,t)=p_{2,m}(t)+c_{2,m}\exp{(\lambda\sigma t)},\hspace{0.25cm}t>0\text{,}
\end{equation}
where $c_{2,m}$ is given by:
\begin{equation}\label{co_IrV2:c_2m}
c_{2,m}=\frac{c\nu\mu p!}{\sigma(\lambda\sigma)^p}\text{,}
\end{equation}
being $c$ the constant given in (\ref{pr_IrV1:c}), and $p_{2,m}(x)$ is the polynomial defined by:
\begin{equation}\label{co_IrV2:p_2m}
p_{2,m}(t)=\left\{\begin{array}{ccc}
              c_{2,3}\left(\frac{\lambda^2\sigma}{\nu\mu}t-1\right)& &\text{if }m=3\\
              c_{2,5}\left(\frac{\lambda^3\sigma^2}{2\nu\mu}t^2-\lambda\sigma t-1\right)& &\text{if }m=5\\
			  \frac{c\lambda}{\sigma}t^p-c_{2,m}\left(\displaystyle\sum_{k=1}^{p-1}\frac{(\lambda\sigma)^k}{k!}t^k+1\right)& &\text{if }m\geq 7
                \end{array}\right.,\hspace{0.25cm}t>0\text{.}
\end{equation}
\end{corollary}

\begin{proof}
It follows by solving the integral in the expression of $u$ given in  (\ref{pr_IrV2:ux}) and the use of the identity (\ref{co_IrV1:int}).
\end{proof}

\begin{corollary}\label{co:IrVinf2}
Let $F$, $h$ and $\Phi=\varphi_2$ defined as in (\ref{th_Ir:F}), (\ref{def:h}) and (\ref{def:phi123}), with $m$ and odd number, and $\sigma=\lambda+\nu\mu$. If $u$ is the solution of the Problem P, given in (\ref{th_Ir:u}), then:

\begin{enumerate}
\item If $\sigma\neq0$ then:
\begin{enumerate}
\item if $m=1$, we have:
\begin{equation}
\displaystyle\lim_{t\to+\infty}u_x(0,t)=\left\{\begin{array}{ccc}
                                            -\infty &\text{if }\sigma>0,&\eta<0\\
											+\infty &\text{if }\sigma>0,&\eta>0\\
											\frac{\eta\lambda}{\sigma}&\text{if }\sigma<0&
                                           \end{array}\right.\text{,}
\end{equation}
\item if $m\geq 3$, we have:
\begin{equation}
\displaystyle\lim_{t\to+\infty}u_x(0,t)=\left\{\begin{array}{ccc}
                                            -\infty& &\text{if }\sigma\eta<0\\
																						                            +\infty& &\text{if }\sigma\eta>0												   \end{array}\right.\text{.}
\end{equation}
\end{enumerate}
\item If $\sigma=0$ then:
\begin{equation}
\displaystyle\lim_{t\to+\infty}u_x(0,t)=\left\{\begin{array}{ccc}
                                            -\infty& &\text{if }\eta>0\\
											+\infty& &\text{if }\eta<0
												\end{array}\right.\text{.}
\end{equation}
\end{enumerate}
\end{corollary}

\begin{proposition}\label{pr:IrV3}
Let $F$, $h$ and $\Phi=\varphi_3$ defined as in (\ref{th_Ir:F}), (\ref{def:h}) and (\ref{def:phi123}), and $\delta=\lambda-\nu\mu$. Then the heat flux at the boundary $x=0$ corresponding to the solution $u$ (see (\ref{th_Ir:u})) of the Problem P is given by:
\begin{enumerate}
\item if $\delta\neq 0$ then:
{\small \begin{equation}\label{pr_IrV3:ux}
u_x(0,t)=\left\{\begin{array}{ccc}
                 \frac{\eta}{\delta}\left(\lambda-\nu\mu\exp{(-\lambda\delta t)}\right)& &\text{if }m=1\\
                 \frac{c\lambda}{\delta}t^{(m-1)/2}-\frac{c(m-1)\nu\mu}{2\delta}\exp{(-\lambda\delta t)}\displaystyle\int_0^t\tau^{(m-3)/2}\exp{(\lambda\delta\tau)}d\tau& &\text{if }m>1
                \end{array}\right.,\hspace{0.25cm}t>0\text{.}
\end{equation}}
\item If $\delta=0$ then:
\begin{equation}\label{pr_IrV3:Ux}
u_x(0,t)=\left\{\begin{array}{ccc}
                 \eta(1+\lambda^2t)& &\text{if }m=1\\
                 c t^{(m-1)/2}+\frac{2c\lambda^2}{m+1}t^{(m+1)/2}& &\text{if }m>1
                \end{array}\right.,\hspace{0.5cm}t>0\text{.}
\end{equation}
\end{enumerate}
\end{proposition}

\begin{proof}
The proof of (\ref{pr_IrV3:ux}) and (\ref{pr_IrV3:Ux}) follows by replacing $\lambda^2$ by $-\lambda^2$ and $\sigma$ by $\delta$ in the proof of Proposition \ref{pr:IrV2}.
\end{proof}

\begin{corollary}\label{co:IrV3}
If in Proposition \ref{pr:IrV2} we consider $\delta\neq 0$ and $m$ an odd number given by $m=2p+1$ with $p\in\mathbb{N}$, then we have:
\begin{equation}\label{co_IrV3:ux}
u_x(0,t)=p_{3,m}(c)+c_{3,m}\exp{(-\lambda\delta t)},\hspace{0.5cm}t>0\text{,}
\end{equation}
where $c_{3,m}$ is given by:
\begin{equation}\label{co_IrV3:c_3m}
c_{3,m}=(-1)^{p-1}\frac{c\nu\mu p!}{\delta(\lambda\delta)^p}\text{,}
\end{equation}
being $c$ the constant given in (\ref{pr_IrV1:c}), and $p_{3,m}(x)$ is the polynomial defined by:
\begin{equation}\label{co_IrV2:p_3m}
p_{3,m}(t)=\left\{\begin{array}{ccc}
              c_{3,3}\left(\frac{\lambda^2\delta}{\nu\mu}t-1\right)& &\text{if }m=3\\
              -c_{3,5}\left(\frac{\lambda^3\delta^2}{2\nu\mu}t^2-\lambda\sigma t+1\right)& &\text{if }m=5\\
			  \frac{c\lambda}{\delta}t^p-c_{3,m}\left(\displaystyle\sum_{k=1}^{p-1}\frac{(-\lambda\delta)^k}{k!}t^k+1\right)& &\text{if }m\geq 7
                \end{array}\right.,\hspace{0.25cm}t>0\text{.}
\end{equation}
\end{corollary}

\begin{proof}
It follows by solving the corresponding integral in the expression (\ref{pr_IrV3:ux}) and the use of the identity (\ref{co_IrV1:int}).
\end{proof}

\begin{corollary}\label{co:IrVinf3}
Let $F$, $h$ and $\Phi=\varphi_3$ defined as in (\ref{th_Ir:F}), (\ref{def:h}) and (\ref{def:phi123}), with $m$ and odd number, and $\delta=\lambda-\nu\mu$. If $u$ is the solution of the Problem P, given in (\ref{th_Ir:u}), then:
\begin{enumerate}
\item If $\delta\neq0$ then:
\begin{enumerate}
\item if $m=1$, we have:
\begin{equation}
\displaystyle\lim_{t\to+\infty}u_x(0,t)=\left\{\begin{array}{ccc}
                                            -\infty &\text{if }\delta<0,&\eta<0\\
											+\infty &\text{if }\delta<0,&\eta>0\\																	                        					 \frac{\eta\lambda}{\delta} &\text{if }\delta>0&
                                           \end{array}\right.\text{,}
\end{equation}
\item if $m=3$ or $m=5$, we have:
\begin{equation}
\displaystyle\lim_{t\to+\infty}u_x(0,t)=\left\{\begin{array}{ccc}
                                            -\infty& &\text{if }\eta<0\\
											+\infty& &\text{if }\eta>0
											   \end{array}\right.
\end{equation}
\end{enumerate}
and
\begin{enumerate}[resume]
\item if $m\geq 7$, we have:
\begin{equation}
\displaystyle\lim_{t\to+\infty}u_x(0,t)=\left\{\begin{array}{ccc}
                                            -\infty &\text{if }\delta\eta<0\\
											+\infty &\text{if }\delta\eta>0\\
											   \end{array}
											   \right.\text{.}
\end{equation}
\end{enumerate}
\item If $\delta=0$ then:
\begin{equation}
\displaystyle\lim_{t\to+\infty}u_x(0,t)=\left\{\begin{array}{ccc}
                                            -\infty& &\text{if }\eta<0\\
											+\infty& &\text{if }\eta>0
											   \end{array}\right.\text{.}
\end{equation}
\end{enumerate}
\end{corollary}

Next result is related to the behaviour of the heat flux $u_x(0,t)$ at the face $x=0$ when $t$ tends to $0^+$, and shows that it is independent on the choice of $\Phi$ as any of the functions given in (\ref{def:phi123}).

\begin{corollary}\label{co:IrV0+}
If $F$, $h$ and $\Phi$ are given as in (\ref{th_Ir:F}), (\ref{def:h}) and any of the expressions in (\ref{def:phi123}), respectively, then:
\begin{equation}\label{co_IrV0+:V0+}
\displaystyle\lim_{t\to0^+}u_x(0,t)=\left\{\begin{array}{ccc}
                                            \eta& &\text{if }m=1\\
											 0& &\text{if }m>1
                                           \end{array}\right.\text{,}
\end{equation}
where $u$ is the solution of the Problem P given in (\ref{th_Ir:u}).
\end{corollary}

\begin{proof}
It follows straightforward by computing the limit for the expression of $u_x(0,t)$ given in Proposition \ref{pr:IrV1}, \ref{pr:IrV2} or \ref{pr:IrV3}, according the definition of $\Phi$.
\end{proof}

We end this section by giving explicit solutions to each Problem P. The proofs of the three following propositions follow from Theorem \ref{th:Ir} and Corollary \ref{co:IrV1}, \ref{co:IrV2} or \ref{co:IrV3}, according to the definition of $\Phi$ (see Appendix \ref{app:2}).

\begin{proposition}\label{pr:Iru1}
If $F$, $h$ and $\Phi=\varphi_1$ are defined as in (\ref{th_Ir:F}), (\ref{def:h}) and (\ref{def:phi123}), where $m$ is an odd number given by $m=2p+1$, with $p\in\mathbb{N}_0$, then the function $u$ defined by:
\begin{equation}\label{pr_Iru1:u}
u(x,t)=u_0(x,t)-\nu\Phi(x)\displaystyle\int_0^tV(\tau)d\tau,\hspace{0.5cm}x\geq0,\,t\geq0
\end{equation}
is a solution to Problem P, where $u_0$ is defined by:
\begin{equation}\label{pr_Iru1:u0}
u_0(x,t)=\frac{\eta}{\sqrt{\pi}}\displaystyle\sum_{k=0}^{p}\binom{m}{2k}\Gamma\left(\frac{2k+1}{2}\right)(4t)^kx^{m-2k},\hspace{0.5cm}x\geq0,\,t\geq0\text{,}
\end{equation}
and $V(t)=u_x(0,t)$ is given by (\ref{co_IrV1:ux}).
\end{proposition}

\begin{remark}
If $m=1$, polynomial $p_{1,m}(x)$ is defined by $p_{1,m}(x)=0$, $x>0$.
\end{remark}

\begin{proposition}\label{pr:Iru2}
If $F$, $h$ and $\Phi=\varphi_2$ are defined as in (\ref{th_Ir:F}), (\ref{def:h}) and (\ref{def:phi123}), where $\sigma\neq0$ and $m$ is an odd number given by $m=2p+1$, with $p\in\mathbb{N}_0$, then the function $u$ defined by:
\begin{equation}\label{pr_Iru2:u}
u(x,t)=u_0(x,t)-\nu\Phi(x)\exp(\lambda^2t)\displaystyle\int_0^tV(\tau)\exp{(-\lambda^2\tau)}d\tau,\hspace{0.5cm}x\geq0,\,t\geq0
\end{equation}
is a solution to Problem P, where $u_0$ and $V(t)=u_x(0,t)$ are given by  (\ref{pr_Iru1:u0}) and (\ref{co_IrV2:ux}).
\end{proposition}

\begin{remark}
If $m=1$, polynomial $p_{2,m}(x)$ is defined by $p_{2,m}(x)=\frac{\nu\lambda}{\sigma}$, $x>0$.
\end{remark}

\begin{proposition}\label{pr:Iru3}
If $F$, $h$ and $\Phi=\varphi_3$ are defined as in (\ref{th_Ir:F}), (\ref{def:h}) and (\ref{def:phi123}), where $\delta\neq0$ and $m$ is an odd number given by $m=2p+1$, with $p\in\mathbb{N}_0$, then the function $u$ defined by:
\begin{equation}\label{pr_Iru3:u}
u(x,t)=u_0(x,t)-\nu\Phi(x)\exp(-\lambda^2t)\displaystyle\int_0^tV(\tau)\exp{(\lambda^2\tau)}d\tau,\hspace{0.5cm}x\geq0,\,t\geq0
\end{equation}
is a solution to Problem P, where $u_0$ and $V(t)=u_x(0,t)$ are given by  (\ref{pr_Iru1:u0}) and (\ref{co_IrV3:ux}).
\end{proposition}

\begin{remark}
If $m=1$, polynomial $p_{3,m}(x)$ is defined by $p_{3,m}(x)=\frac{\nu\lambda}{\delta}$, $x>0$.
\end{remark}

\section{The controlling problem}\label{s:Cp}
This section is devoted to study the effects introduced by the source term $-\Phi F$ in the asymptotic behaviour of the solution $u$ to each Problem P considered in this paper. We will carry out our study by comparing the asymptotic behaviour of $u$ with the asymptotic behaviour of the solution $u_0$ to Problem P in the absence of control (Problem P$_0$):
\begin{align}%\label{pb:P}
\label{eq:u}&u_t(x,t)-u_{xx}(x,t)=0\hspace{0.5cm}	&x>0,\hspace{0.5cm}&t>0 \\
\label{ic:u}&u(x,0)=h(x)\hspace{0.5cm}	&x>0\hspace{0.5cm}&\\
\label{bc:u}&u(0,t)=0	&	&t>0
\end{align}

\noindent This kind of analysis enable us to control Problem P by its source term.

The study of controlling Problem P by its source term has been done in \cite{BeTaVi2000} when $\Phi$ is identically equal to 1, $F=F(V)$ is a differentiable function of one real variable which satisfies:
\begin{enumerate}
\item $VF(V)\geq0,\quad\forall\,V\in\mathbb{R}$,
\item $F(0)=0$,
\item $F$ is convex in $(0,+\infty)$,
\item $\displaystyle\lim_{V\to+\infty}F'(V)=\kappa>0$,
\end{enumerate}
and $h$ is a non-negative, continuous and bounded function. They proved that under these hypothesis, both $u$ and $u_0$ converge to 0 when $t$ tends to $+\infty$ and the control term $F$ has a stabilizing effect because $\displaystyle\lim_{t\to+\infty}\frac{u(x,t)}{u_0(x,t)}=0$, that is, $u$ converge faster to 0 than $u_0$. None of the cases studied in the previous sections fulfil the hypothesis for $\Phi$, $F$ and $h$ established in \cite{BeTaVi2000}.

With the aim of supplementing the results given in \cite{BeTaVi2000}, we will carry out our analysis under conditions which lead us to functions $F$ depending on only one real variable, that is $F=F(V)$.

Next Theorems \ref{th:lim-In}, \ref{th:lim-Sv} and \ref{th:lim-Ir} are respectively related with the results obtained in Sections \ref{ss:In}, \ref{ss:Sv} and \ref{ss:Ir}.

\begin{remark}
For all Problems P studied in this paper, Problem P$_0$ has the solution $u_0$ defined by \cite{Ca1984}:
\begin{equation}\label{u_0}
u_0(x,t)=\displaystyle\int_0^{+\infty}G(x,t,\xi,0)h(\xi)d\xi,\hspace{0.5cm}x\geq0,\,t\geq0\text{,}
\end{equation}
where $G$ is the Green function defined in (\ref{th_Ir:G}).
\end{remark}

\begin{theorem}\label{th:lim-In}
Let $\Phi$ identically equal to 1, $F$ a constant function defined by:
\begin{equation}
F(V)=\nu,\hspace{0.5cm}V\in\mathbb{R}\text{,}
\end{equation}
for some $\nu\in\mathbb{R}-\{0\}$, and $h$ a quadratic function defined by:
\begin{equation}\label{th_lim-In:h}
h(x)=\frac{\nu}{2}x^2+ax,\hspace{0.5cm}x\geq0\text{,}
\end{equation}
for some $a\in\mathbb{R}$.

For the solution $u_0$ to Problem P$_0$ given in (\ref{u_0}), we have:
\begin{equation}\label{th_lim-In:u_0}
\displaystyle\lim_{t\to+\infty}u_0(x,t)=\infty\text{,}\hspace{0.5cm}\forall\,x>0\text{.}
\end{equation}
Furthermore, there exists a solution $u$ to Problem P such that:
\begin{equation}\label{th_lim-In:u}
\displaystyle\lim_{t\to+\infty}u(x,t)=h(x)\text{,}\hspace{0.3cm}\forall\,x>0\text{.}
\end{equation}
%Therefore,
%\begin{equation}\label{th_lim-In:u/u_0}
%\displaystyle\lim_{t\to+\infty}\frac{u(x,t)}{u_0(x,t)}=0\text{,}\hspace{0.7cm}\forall\,x>0\text{.}
%\end{equation}
\end{theorem}

\begin{proof}
By computing the integral in (\ref{u_0}) for the function $h$ given in (\ref{th_lim-In:h}), we have that the solution $u_0$ to Problem P$_0$ given in (\ref{u_0}) is defined by:
\begin{equation}\label{th_lim-In:uu_0}
u_0(x,t)=\left(\frac{\nu}{2}x^2+\nu t\right)\text{erf}\left(\frac{x}{2\sqrt{t}}\right)+\frac{\nu}{\sqrt{\pi}}x\sqrt{t}\exp{\left(\frac{-x^2}{4t}\right)}+ax,\hspace{0.5cm}x>0,\,t>0\text{.}
\end{equation}
By taking the limit when $t$ tends to $+\infty$, we have (\ref{th_lim-In:u_0}).

Since functions $\Phi$, $F$ and $h$ are under the hypothesis of Theorem \ref{th:In}, we know that the function $u$ given by:
\begin{equation}\label{th_lim-In:uu}
u(x,t)=h(x),\hspace{0.5cm}x\geq0,\,t\geq0\text{,}
\end{equation}
is a solution to Problem P, which satisfies (\ref{th_lim-In:u}).
%Finally, the proof of (\ref{th_lim-In:u/u_0}) follows straightforward by computing the limit from the explicit expressions of $u_0$ and $u$ given in (\ref{th_lim-In:uu_0}) and (\ref{th_lim-In:uu}).
\end{proof}

\begin{theorem}\label{th:lim-Sv}
Let $\Phi$, $h$ and $F$ defined by:
\begin{equation}\label{th_lim-Sv:h-Phi-F}
\Phi(x)=\lambda X(x),\,\,h(x)=\eta X(x),\,x>0\hspace{0.5cm}\text{and}\hspace{0.5cm}F=F(\delta T(t)),\,t>0\text{,}
\end{equation}
where $X$ is the function given by (\ref{th_Sv:X}):
\begin{equation*}
X(x)=\left\{\begin{array}{ccc}
             \frac{\delta}{\sqrt{\sigma}}\sinh{(\sqrt{\sigma}x)} & &\text{if }\sigma>0\\
             \frac{\delta}{\sqrt{|\sigma|}}\sin{(\sqrt{|\sigma|}x)}& & \text{if }\sigma<0\\
			 \delta x& &\text{if }\sigma=0
            \end{array}\right.,\hspace{0.5cm}x>0
\end{equation*}
and $T$ is the solution of the initial value problem (\ref{th_Sv:T})-(\ref{th_Sv:TT}):
\begin{align*}
&\dot{T}(t)-\sigma T(t)=-\lambda F(\delta T(t),t),\hspace{0.5cm}t>0\\
&T(0)=\eta \text{,}
\end{align*}
with $\lambda,\,\eta,\,\delta\in\mathbb{R}-\{0\}$.

For the solution $u_0$ to Problem P$_0$ given in (\ref{u_0}), we have:
\begin{equation}\label{th_lim-Sv:u_0}
\displaystyle\lim_{t\to+\infty}u_0(x,t)=\left\{\begin{array}{ccc}
                                             h(x)& &\text{if }\sigma=0\\
                                             \infty& &\text{if }\sigma> 0\\
                                             0& &\text{if }\sigma< 0\\
                                             \end{array}\right.\text{,}\hspace{0.5cm}\forall\,x>0\text{.}
\end{equation}
Furthermore:
\begin{enumerate}
\item If $F$ is defined by:
\begin{equation}\label{th_lim-Sv:F}
F(V)=\nu V,\hspace{0.5cm}V\in\mathbb{R}\text{,}
\end{equation}
for some $\nu\in\mathbb{R}-\{0\}$, then there exists a solution $u$ to Problem P which satisfies:
\begin{equation}\label{th_lim-Sv:u}
\displaystyle\lim_{t\to+\infty}u(x,t)=\left\{\begin{array}{ccccc}
                                             h(x)& &\text{if }\gamma=\sigma& & \\
                                             \infty& &\text{if }\gamma<\sigma&\text{,} &\forall\,x>0 \text{ if } \gamma\geq\sigma\\
                                             0& &\text{if }\gamma>\sigma& &\forall\,x>0\,/\,h(x)\neq0 \text{ if } \gamma<\sigma
                                             \end{array}\right.\text{,}
\end{equation}
being $\gamma=\lambda\nu\delta$.

Therefore,
\begin{equation}\label{th_lim-Sv:u/u_0}
\displaystyle\lim_{t\to+\infty}\frac{u(x,t)}{u_0(x,t)}=
\left\{\begin{array}{ccccc}
        \infty& &\text{if }\gamma<0&\text{,} &\forall\,x>0\text{ if }\sigma\geq0\\
        0& &\text{if }\gamma>0& &\forall\,x>0\,/\,h(x)\neq0\text{ if }\sigma<0
       \end{array}\right.\text{.}
\end{equation}
\item If $F$ is defined by:
\begin{equation}\label{th_lim-Sv:FF}
F(V)=\nu V^n,\hspace{0.5cm}V\in\mathbb{R}\text{,}
\end{equation}
for some $\nu>0$ and $n<1$, and we consider $\lambda>0$, $\eta>0$ and $\delta>0$, then there exists a solution $u$ to Problem P which satisfies:
\begin{equation}\label{th_lim-Sv:uu}
\displaystyle\lim_{t\to+\infty}u(x,t)=\left\{\begin{array}{ccccc}
                                             0& &\text{if }\sigma< 0 \vee 0<n<1& &\forall\,x>0\text{ if }\sigma<0\,\wedge\,0\leq n<1\\
                                             \theta_1h(x)& &\text{if }\sigma<0 \wedge n=0&\text{,} &\forall\,x>0\,/\,h(x)\neq0\text{ if }\\
                                             \infty& &\text{if }\sigma\geq 0 \vee (\sigma<0 \wedge n<0)& &\sigma\geq0\,\vee\,(n<0\,\wedge\,\sigma<0)
                                             \end{array}\right.\text{,}
\end{equation}
where $\theta_1=\frac{\lambda\mu}{\sigma\eta}$.

Therefore,
\begin{equation}\label{th_lim-Sv:uu/uu_0}
\displaystyle\lim_{t\to+\infty}\frac{u(x,t)}{u_0(x,t)}=
\left\{\begin{array}{ccccc}
        \infty& &\text{if }\sigma\leq0&\text{,} &\forall\,x>0\text{ if }\sigma\geq0\\
        1-\frac{\lambda\mu\delta^n}{\sigma\eta^{1-n}}& &\text{if }\sigma>0& &\forall\,x>0\,/\,h(x)\neq0\text{ if }\sigma<0
       \end{array}\right.\text{.}
\end{equation}
\end{enumerate}
\end{theorem}

\begin{proof}
By computing the integral in (\ref{u_0}) for the function $h$ given in (\ref{th_lim-Sv:h-Phi-F}), we obtain that the solution $u_0$ to Problem P$_0$ given in (\ref{u_0}) is defined by:
\begin{equation}\label{th_lim-Sv:uuu_0}
u_0(x,t)=h(x)\exp{(|\sigma|t)},\hspace{0.5cm}\forall\,x>0,\,t>0\text{.}
\end{equation}
By taking the limit when $t$ tends to $+\infty$, we have (\ref{th_lim-Sv:u_0}).
\begin{enumerate}
\item Since $\Phi$, $h$ and $F$ are under the hypothesis of Corollary \ref{pr:Sv}, we know that the function $u$ given by:
\begin{equation}\label{th_lim-Sv:uuu}
u(x,t)=\eta X(x)\exp{((\sigma-\gamma)t)},\hspace{0.5cm}\forall\,x>0,\,t>0
\end{equation}
is a solution to Problem P, which satisfies (\ref{th_lim-Sv:u}).
Finally, the proof of (\ref{th_lim-Sv:u/u_0}) follows straightforward by computing the limit from the explicit expressions of $u_0$ and $u$ given in (\ref{th_lim-Sv:uuu_0}) and (\ref{th_lim-Sv:uuu}).
\item It follows in the same manner that the proof of the previous item.
\end{enumerate}
\end{proof}

We see from the previous theorem that we can control the Problem P through the parameters involved in the definition of the source term $-\Phi F$. When $F(V)=\nu V$, we can increase ($\gamma<0<\sigma$) or decrease ($0<\gamma<\sigma$) the velocity of convergence to $\infty$ for $u$ respect to the velocity of convergence for $u_0$. We also can stabilize the problem by doing $u$ tending to a constant value ($0<\sigma\leq\gamma$) when $u_0$ is going to $\infty$. When $F(V)=\nu V^n$, we can decrease ($\sigma>0$ and $1=\frac{\lambda\mu\delta^n}{\sigma\eta^{1-n}}$) or maintein ($\sigma>0$ and $1\neq\frac{\lambda\mu\delta^n}{\sigma\eta^{1-n}}$) the velocity of convergence to $\infty$ for $u$ respect to the velocity of convergence for $u_0$. We also can decrease the velocity of convergence to 0 for $u$ respect to the velocity of convergence for $u_0$ ($\sigma<0$).

\begin{theorem}\label{th:lim-Ir}
Let $\Phi$ defined by one of the expressions given in (\ref{def:phi123}):
\begin{equation*}
\varphi_1(x)=\lambda x,\hspace{0.5cm}\varphi_2(x)=-\mu\sinh{(\lambda x)}\hspace{0.5cm}\text{or}\hspace{0.5cm}\varphi_3(x)=-\mu\sin{(\lambda x)},\hspace{0.5cm}x>0\text{,}
\end{equation*}
where $\lambda>0$ and $\mu>0$, $F$ defined by:
\begin{equation}
F=F(V)=\nu V,\hspace{0.5cm} V\in\mathbb{R}\text{,}
\end{equation}
for some $\nu>0$ and $h$ defined as in (\ref{def:h}):
\begin{equation*}
h(x)=\eta x^m,\hspace{0.5cm}x>0\text{,}
\end{equation*}
where $\eta\in\mathbb{R}-\{0\}$ and $m$ is an odd number given by $m=2p+1$, with $p\in\mathbb{N}_0$.

For the solution $u_0$ to Problem P$_0$ given in (\ref{u_0}), we have:
\begin{equation}\label{th_lim-Ir:u_0}
\displaystyle\lim_{t\to+\infty}u_0(x,t)=\left\{\begin{array}{ccc}
                                             h(x)& &\text{if }m=1\\
                                             \infty& &\text{if }m>1
                                             \end{array}\right.\text{,}\hspace{0.5cm}\forall\,x>0\text{.}
\end{equation}
Furthermore:
\begin{enumerate}
\item If $\Phi=\varphi_1$, then there exists a solution $u$ to Problem P which satisfies:
\begin{equation}
\displaystyle\lim_{t\to+\infty}u(x,t)=\left\{\begin{array}{ccc}
                                             0& &\text{if }m=1\\
                                             \infty& &\text{if }m>1
                                             \end{array}\right.\text{,}\hspace{0.5cm}\forall\,x>0\text{.}
\end{equation}
Therefore,
\begin{equation}
\displaystyle\lim_{t\to+\infty}\frac{u(x,t)}{u_0(x,t)}=\left
                                             \{\begin{array}{ccc}
                                             0& &\text{if }m=1\\
                                             r(x)& &\text{if }m>1
                                             \end{array}\right.\text{,}\hspace{0.5cm}\forall\,x>0\text{,}
\end{equation}
being $r(x)$ is a rational function in the variable $x$.
\item If $\Phi=\varphi_2$, then there exists a solution $u$ to Problem P which satisfies:
\begin{equation}
\displaystyle\lim_{t\to+\infty}u(x,t)=\infty\text{,}\hspace{0.5cm}\forall\,x>0\text{.}
\end{equation}
Therefore,
\begin{equation}
\displaystyle\lim_{t\to+\infty}\frac{u(x,t)}{u_0(x,t)}=\infty\text{,}\hspace{0.5cm}\forall\,x>0\text{.}
\end{equation}
\item If $\Phi=\varphi_3$, then there exists a solution $u$ to Problem P which satisfies:
\begin{equation}
\displaystyle\lim_{t\to+\infty}u(x,t)=\infty\text{,}\hspace{3.8cm}\forall\,x>0\text{.}
\end{equation}
Therefore,
\begin{equation}
\displaystyle\lim_{t\to+\infty}\frac{u(x,t)}{u_0(x,t)}=\left\{
                                             \begin{array}{ccc}
                                             r(x)& &\text{if }\delta>0\\
                                             \infty& &\text{if }\delta\leq 0
                                             \end{array}\right.\text{,}\hspace{0.5cm}\forall\,x>0\text{,}
\end{equation}
where $r(x)$ is a rational function in the variable $x$.
\end{enumerate}
\end{theorem}

\begin{proof}
It follows in the same manner that the proofs of Theorems \ref{th:lim-In} and \ref{th:lim-Sv}.
\end{proof}

From the previous theorem, we see again that there exist several cases where we can control the Problem P through the source term $-\Phi F$.

\section{Explicit solutions for Problem $\widetilde{\text{P}}$}\label{s:EsPP}
The following theorem states a relationship between the Problem P and the Problem $\widetilde{\text{P}}$ given in (\ref{eq:v})-(\ref{bc:v}), and it was proved in \cite{TaVi1998}.

\begin{theorem}\label{th:PP}
If $u$ is a solution to Problem P where $h$ and $\Phi$ are differentiable functions in $\mathbb{R}^+$, then the function $v$ defined by:
\begin{equation}
v(x,t)=u_x(x,t),\hspace{0.5cm}x\geq0,\,t\geq 0
\end{equation}
is a solution to Problem $\widetilde{\text{P}}$ when $\widetilde{F}$, $\widetilde{\Phi}$, $\tilde{h}$ and $\tilde{g}$ are defined by:
\begin{equation}
\begin{split}
&\widetilde{F}(V,t)=F(V,t),\hspace{0.25cm} V>0,\,t>0,\hspace{1cm}\tilde{g}(t)=\Phi(0)F(u_x(0,t),t),\hspace{0.25cm}t>0\text{,}\\
&\widetilde{\Phi}(x)=\Phi'(x),\hspace{0.25cm}x>0,\hspace{2.9cm}\tilde{h}(x)=h'(x),\hspace{0.25cm} x>0\text{.}
\end{split}
\end{equation}
\end{theorem}

We end this section by giving explicit solutions for some particular cases of Problem $\widetilde{\text{P}}$.

\begin{proposition}\label{pr:PP}
Let $\tilde{g}$ the zero function and:
\begin{enumerate}
\item
\begin{enumerate}
\item $\widetilde{F}$ the zero function and $\tilde{h}$ a constant function, or
\item $\widetilde{F}$ a constant function defined by:
\begin{equation}
\widetilde{F}(V,t)=k,\hspace{0.5cm}V\in\mathbb{R},\,t>0\text{,}
\end{equation}
for some $k\in\mathbb{R}-\{0\}$, $\widetilde{\Phi}$ a locally integrable function in $\mathbb{R}^+$ and $\tilde{h}$ a differentiable function such that:
\begin{equation}
\tilde{h}(x)=k\displaystyle\int_0^x\widetilde{\Phi}(\xi)d\xi,\hspace{0.5cm}x>0\text{.}
\end{equation}
Then the function $v$ defined by:
\begin{equation}
v(x,t)=\tilde{h}(x),\hspace{0.5cm}x\geq0,\,t\geq0
\end{equation}
is a solution to Problem $\widetilde{\text{P}}$ independent of the temporal variable $t$.
\end{enumerate}
\item $\widetilde{F}$ given by (\ref{pr_Sv:F}), (\ref{pr_Sv:FF}) or (\ref{pr_Sv:FFF}), that is:
\begin{enumerate}
\item[]$F(V,t)=\nu V,\hspace{0.5cm}V\in\mathbb{R},\,t>0,\hspace{0.25cm}$with $\nu\in\mathbb{R}-\{0\}$,
\item[]$F(V,t)=f_1(t)+f_2(t)V,\hspace{0.5cm}V\in\mathbb{R},\,t>0,\hspace{0.25cm}$with $f_1, f_2\in L^1_{loc}(\mathbb{R}^+)$, or
\item[]$F(V,t)=V^nf(t),\hspace{0.5cm}V\in\mathbb{R},\,t>0,\hspace{0.25cm}$with $n<1, f\in L^1_{\text{loc}}(\mathbb{R}^+), f>0$ and $\lambda, \delta, \eta>0$,
\end{enumerate}
and $\tilde{h}$ and $\widetilde{\Phi}$ defined by:
\begin{equation}
\tilde{h}(x)=\eta \widetilde{X}(x)\hspace{0.5cm}\text{ and }\hspace{0.5cm}\widetilde{\Phi}(x)=\lambda \widetilde{X}(x),\hspace{0.5cm}t>0\text{,}
\end{equation}
where $\widetilde{X}$ is given by:
\begin{equation}
\widetilde{X}(x)=\left\{\begin{array}{ccc}
                    \delta\cosh{(\sqrt{\sigma})x}& &\text{if }\sigma>0\\
                    \delta\cos{(\sqrt{|\sigma|})x}& &\text{if }\sigma<0\\
                    \delta& &\text{if }\sigma=0\\
                    \end{array}\right.,\hspace{0.5cm}x>0\text{,}
\end{equation}
$\lambda,\, \eta,\, \delta\in\mathbb{R}-\{0\}$.

Then the function $v$ defined by:
\begin{equation}
v(x,t)=\widetilde{X}(x)\widetilde{T}(t),\hspace{0.5cm}x\geq0,\,t\geq0
\end{equation}
is a solution with separated variables to Problem $\widetilde{\text{P}}$, where $\widetilde{T}$ the solution of the initial value problem (\ref{th_Sv:T})-(\ref{th_Sv:TT}).
\item $\widetilde{F}$ defined as in (\ref{th_Ir:F}):
\begin{equation*}
F=F(V,t)=\nu V,\hspace{0.5cm} V\in\mathbb{R},\,t>0\text{,}
\end{equation*}
for some $\nu>0$, $\tilde{h}$ defined as:
\begin{equation}
\tilde{h}(x)=\tilde{\eta}x^l,\hspace{0.5cm}x>0\text{,}
\end{equation}
for some $\tilde{\eta}\in\mathbb{R}-\{0\}$ and $l\geq 0$, and $\widetilde{\Phi}$ given by one of the following expressions:
\begin{equation}
\widetilde{\varphi}_1(x)=\tilde{\lambda}x,\hspace{0.5cm}\varphi_2(x)=-\tilde{\mu}\cosh{(\tilde{\lambda}x)}\hspace{0.5cm}\text{or}\hspace{0.5cm}\varphi_3(x)=-\tilde{\mu}\cos{(\tilde{\lambda}x)},\hspace{0.5cm}x>0\text{,}
\end{equation}
for some $\tilde{\lambda}>0$ and $\tilde{\nu}>0$. Then the function $v$ defined by:
\begin{equation}
v(x,t)=u_x(x,t),\hspace{0.5cm}x\geq0,\,t\geq0
\end{equation}
is a solution to Problem $\widetilde{\text{P}}$, where $u$ is given by (\ref{pr_Iru1:u}) if $\widetilde{\Phi}=\tilde{\varphi}_1$, by  (\ref{pr_Iru2:u}) if $\widetilde{\Phi}=\tilde{\varphi}_2$ or by (\ref{pr_Iru3:u}) if $\widetilde{\Phi}=\tilde{\varphi}_3$.
\end{enumerate}
\end{proposition}

\begin{proof}
It follows from the previous theorem and the explicit solutions to Problem P obtained in Section \ref{s:EsP}.
\end{proof}

\section*{Acknowledgment}
This paper has been partially sponsored by the Project PIP No. 0534 from CONICET-UA and Grant from Universidad Austral (Rosario, Argentina).

%-------------------------------------------------------------------------------------
%\nocite{*}
\bibliographystyle{plain}
\bibliography{references}

%-------------------------------------------------------------------------------------
\appendix

\section*{Appendices}

\section{Problem P is under the hypothesis of Theorem \ref{th:Ir}}\label{app:1}
Let a Problem P with $F$ and $h$ given as in (\ref{th_Ir:F}) and (\ref{def:h}), respectively, and $\Phi$ defined by any of the expressions $\varphi_1$, $\varphi_2$ or $\varphi_3$ given in (\ref{def:phi123}).
\begin{enumerate}
\item It is clear that $h$ is a continuously differentiable function such that $h(0)$ exists. We also have:
\begin{equation}
|h(x)|=|\eta|x^m\leq |\eta|(x+1)^m\leq |\eta|\exp{(mx)},\hspace{0.5cm}\forall\,x>0\text{.}
\end{equation}
Then the inequality (\ref{th_Ir:h}) holds with $\epsilon=1$, $c_1=m$ and $c_0=|\eta|$.
\item It is easy to check that each of the functions $\varphi$ given in (\ref{def:phi123}) is uniformily H$\ddot{\text{o}}$lder continous, with H$\ddot{\text{o}}$lder exponent $\alpha=1$, on any compact set $K\subset\mathbb{R}$.
\item Hypothesis 3 holds because of the definition of the Problem P.
\end{enumerate}

Furthermore, we have:
\begin{enumerate}
\item[i.] If $\Phi=\varphi_1$, then:
\begin{equation}
R(t)=\lambda,\hspace{0.5cm}t>0\text{.}
\end{equation}
Then, the inequality (\ref{th_Ir:f}) holds with a function $f$ defined by:
\begin{equation}
f(t)=-\lambda t,\hspace{0.5cm}t>0\text{.}
\end{equation}
\item[ii.] If $\Phi=\varphi_2$, then:
\begin{equation}
R(t)=-\lambda\mu\exp{(\lambda^2t)},\hspace{0.5cm}t>0\text{.}
\end{equation}
Then the inequality (\ref{th_Ir:f}) holds with a function $f$ defined by:
\begin{equation}
f(t)=-\frac{\mu}{\lambda}\left(\exp{(\lambda^2t)-1}\right),\hspace{0.5cm}t>0\text{.}
\end{equation}
\item[iii.] If $\Phi=\varphi_3$, then:
\begin{equation}
R(t)=-\lambda\mu\exp{(-\lambda^2t)},\hspace{0.5cm}t>0\text{.}
\end{equation}
Then the inequality (\ref{th_Ir:f}) holds with a function $f$ defined by:
\begin{equation}
f(t)=-\frac{\mu}{\lambda}\left(1-\exp{(-\lambda^2t)}\right),\hspace{0.5cm}t>0\text{.}
\end{equation}
\end{enumerate}

\section{Proof of Propositions \ref{pr:Iru1}, \ref{pr:Iru2} and \ref{pr:Iru3}}\label{app:2}
Let a Problem P with $F$ and $h$ given as in (\ref{th_Ir:F}) and (\ref{def:h}), respectively, and $\Phi$ given by any of the expressions in (\ref{def:phi123}).
\subsection*{Computation of $\displaystyle\int_0^{+\infty}G(x,t,\xi,0)h(\xi)d\xi$}
By the definitions of the functions $G$ and $h$ given in (\ref{th_Ir:G}) and (\ref{def:h}), respectively, we have:
\begin{equation}
\displaystyle\int_0^{+\infty}G(x,t,\xi,0)h(\xi)d\xi=\frac{\eta}{2\sqrt{\pi t}}\displaystyle\int_0^{+\infty}\left(\exp{\left(-(x-\xi)^2/4t\right)}-\exp{\left(-(x+\xi)^2/4t\right)}\right)\xi^md\xi\text{.}
\end{equation}
We first compute $\displaystyle\int_0^{+\infty}\exp{\left(-(x-\xi)^2/4t\right)}\xi^md\xi$.

\noindent By doing the substitution $\zeta=(x-\xi)/2\sqrt{t}$, we have:
\begin{equation}
\begin{split}
\displaystyle\int_0^{+\infty}\exp{\left(-(x-\xi)^2/4t\right)}\xi^md\xi&=2\sqrt{t}\displaystyle\int_{-\infty}^{x/2\sqrt{t}}\exp{\left(-\zeta^2\right)}\left(x-2\sqrt{t}\zeta\right)^md\zeta\\
&=2\sqrt{t}\displaystyle\sum_{k=0}^m\binom{m}{k}\left(-2\sqrt{t}\right)^kx^{m-k}\displaystyle\int_{-\infty}^{x/2\sqrt{t}}\exp{\left(-\zeta^2\right)\zeta^k}d\zeta\text{.}
\end{split}
\end{equation}
Since:
\begin{equation}
\begin{split}
\displaystyle\int_{-\infty}^{x/2\sqrt{t}}\exp{\left(-\zeta^2\right)\zeta^k}d\zeta&=
\displaystyle\int_{-\infty}^0\exp{\left(-\zeta^2\right)\zeta^k}d\zeta+\displaystyle\int_0^{x/2\sqrt{t}}\exp{\left(-\zeta^2\right)\zeta^k}d\zeta\\
&=(-1)^k\displaystyle\int_0^{+\infty}\exp{\left(-\sigma^2\right)\sigma^k}d\sigma+\displaystyle\int_0^{x/2\sqrt{t}}\exp{\left(-\zeta^2\right)\zeta^k}d\zeta\\
&=\frac{(-1)^k}{2}\Gamma\left(\frac{k+1}{2}\right)+\displaystyle\int_0^{x/2\sqrt{t}}\exp{\left(-\zeta^2\right)\zeta^k}d\zeta\text{,}
\end{split}
\end{equation}
then we have:
\begin{equation}
\begin{split}
\displaystyle\int_0^{+\infty}\exp{\left(-(x-\xi)^2/4t\right)}\xi^md\xi&=
\sqrt{t}\displaystyle\sum_{k=0}^m\binom{m}{k}\left(2\sqrt{t}\right)^kx^{m-k}\Gamma\left(\frac{k+1}{2}\right)+\\
&2\sqrt{t}\displaystyle\sum_{k=0}^m\binom{m}{k}\left(-2\sqrt{t}\right)^kx^{m-k}\displaystyle\int_0^{x/2\sqrt{t}}\exp{\left(-\zeta^2\right)\zeta^k}d\zeta\text{.}
\end{split}
\end{equation}

\noindent By similar calculations, we have:
\begin{equation}
\begin{split}
\displaystyle\int_0^{+\infty}\exp{\left(-(x+\xi)^2/4t\right)}\xi^md\xi&=
\sqrt{t}\displaystyle\sum_{k=0}^m(-1)^{m-k}\binom{m}{k}\left(2\sqrt{t}\right)^kx^{m-k}\Gamma\left(\frac{k+1}{2}\right)+\\
&2\sqrt{t}\displaystyle\sum_{k=0}^m(-1)^{m-k}\binom{m}{k}\left(-2\sqrt{t}\right)^kx^{m-k}\displaystyle\int_0^{x/2\sqrt{t}}\exp{\left(-\zeta^2\right)\zeta^k}d\zeta\text{.}
\end{split}
\end{equation}
Therefore, we have:
\begin{equation}\label{app_2:Gh}
\begin{split}
\displaystyle\int_0^{+\infty}G(x,t,\xi,0)h(\xi)d\xi&=\frac{\eta}{2\sqrt{\pi t}}\left(\sqrt{t}\displaystyle\sum_{k=0}^m\left(1-(-1)^{m-k}\right)\binom{m}{k}\left(2\sqrt{t}\right)^kx^{m-k}\Gamma\left(\frac{k+1}{2}\right)\right.+\\
&\left.2\sqrt{t}\displaystyle\sum_{k=0}^m\left(1-(-1)^{m+1}\right)\binom{m}{k}\left(-2\sqrt{t}\right)^kx^{m-k}\displaystyle\int_0^{x/2\sqrt{t}}\exp{\left(-\zeta^2\right)\zeta^k}d\zeta\right)\\
&=\frac{\eta}{2\sqrt{\pi}}\displaystyle\sum_{k=0}^m\left(1-(-1)^{k+1}\right)\binom{m}{k}\left(2\sqrt{t}\right)^kx^{m-k}\Gamma\left(\frac{k+1}{2}\right)\\
&=\frac{\eta}{\sqrt{\pi}}\displaystyle\sum_{k=0}^{p}\binom{m}{2k}\left(4t\right)^kx^{m-2k}\Gamma\left(\frac{2k+1}{2}\right)\text{,}\hspace{0.5cm}x>0\text{.}
\end{split}
\end{equation}
\subsection*{Computation of $\displaystyle\int_0^{+\infty}G(x,t,\xi,\tau)\Phi(\xi)d\xi$}
\begin{enumerate}
\item By the definitions of the functions $G$ and $\Phi=\varphi_1$ given in (\ref{th_Ir:G}) and (\ref{def:phi123}), respectively, we have:
\begin{equation}
\begin{split}
\displaystyle\int_0^{+\infty}G(x,t,\xi,\tau)\varphi_1(\xi)d\xi&=\frac{\lambda}{2\sqrt{\pi(t-\tau)}}\displaystyle\int_0^{+\infty}\left(\exp{\left(-(x-\xi)^2/4(t-\tau)\right)}\right.-\\
&\left.\exp{\left(-(x+\xi)^2/4(t-\tau)\right)}\right)\xi d\xi\text{.}
\end{split}
\end{equation}
By replacing $t$ by $(t-\tau)$, $\eta$ by $\lambda$ and $m$ by 1 in the precedent calculation, we have:
\begin{equation}\label{app_2:Gphi1}
\displaystyle\int_0^{+\infty}G(x,t,\xi,\tau)\varphi_1(\xi)d\xi=\varphi_1(x)\text{,}\hspace{0.5cm}x>0\text{.}
\end{equation}

\item By the definitions of the functions $G$ and $\Phi=\varphi_2$ given in (\ref{th_Ir:G}) and (\ref{def:phi123}), respectively, we have:
\begin{equation}
\begin{split}
\displaystyle\int_0^{+\infty}G(x,t,\xi,\tau)\varphi_2(\xi)d\xi&=-\frac{\mu}{2\sqrt{\pi(t-\tau)}}\displaystyle\int_0^{+\infty}\left(\exp{\left(-(x-\xi)^2/4(t-\tau)\right)}\right.-\\
&\left.\exp{\left(-(x+\xi)^2/4(t-\tau)\right)}\right)\sinh{(\lambda\xi)} d\xi\text{.}
\end{split}
\end{equation}
We first compute $\displaystyle\int_0^{+\infty}\exp{\left(-(x-\xi)^2/4(t-\tau)\right)}\sinh{(\lambda\xi)}d\xi$.

By doing the change of variables $\zeta=(x-\xi)/2\sqrt{t-\tau}$, we have:
\begin{multline}
\displaystyle\int_0^{+\infty}\exp{\left(-(x-\xi)^2/4(t-\tau)\right)}\exp{(\lambda\xi)}d\xi=\\
2\sqrt{t-\tau}\exp{(\lambda x)}\displaystyle\int_{-\infty}^{x/2\sqrt{t-\tau}}\exp{\left(-\zeta^2-2\lambda\sqrt{t-\tau}\zeta\right)}d\zeta\text{.}
\end{multline}
By writing:
\begin{equation}
\zeta^2+2\lambda\sqrt{t-\tau}\zeta=\left(\zeta+\lambda\sqrt{t-\tau}\right)^2-\lambda^2(t-\tau)
\end{equation}
and doing the change of variables $\sigma=\zeta+\lambda\sqrt{t-\tau}$, we have:
\begin{equation}
\begin{split}
&2\sqrt{t-\tau}\exp{(\lambda x)}\displaystyle\int_{-\infty}^{x/2\sqrt{t-\tau}}\exp{(-\zeta^2-2\lambda\sqrt{t-\tau}\zeta)}d\zeta=\\
&=2\sqrt{t-\tau}\exp{\left(\lambda x+\lambda^2(t-\tau)\right)}\displaystyle\int_{-\infty}^{x/2\sqrt{t-\tau}}\exp{\left(-\left(\zeta+\lambda\sqrt{t-\tau}\right)^2\right)}d\zeta\\
&=2\sqrt{t-\tau}\exp{\left(\lambda x+\lambda^2(t-\tau)\right)}\displaystyle\int_{-\infty}^{x/2\sqrt{t-\tau}+\lambda\sqrt{t-\tau}}\exp{(-\sigma^2)}d\sigma\\
&=\sqrt{\pi(t-\tau)}\exp{\left(\lambda x+\lambda^2(t-\tau)\right)}\left(1+\text{erf}\left(\frac{x}{2\sqrt{t-\tau}}+\lambda\sqrt{t-\tau}\right)\right)\text{,}
\end{split}
\end{equation}
where erf is the error function, defined by:
\begin{equation}
\text{erf}(z)=\displaystyle\int_0^z\exp{(-\xi^2)}d\xi,\hspace{0.5cm}z\in\mathbb{R}\text{.}
\end{equation}
Hence, we have:
\begin{multline}
\displaystyle\int_0^{+\infty}\exp{\left(-(x-\xi)^2/4(t-\tau)\right)}\exp{(\lambda\xi)}d\xi=\\
\sqrt{\pi(t-\tau)}\exp{\left(\lambda x+\lambda^2(t-\tau)\right)}\left(1+\text{erf}\left(\frac{x}{2\sqrt{t-\tau}}+\lambda\sqrt{t-\tau}\right)\right)\text{.}
\end{multline}
By replacing $\lambda$ by $-\lambda$ in the previous calculations, we have:
\begin{multline}
\displaystyle\int_0^{+\infty}\exp{\left(-(x-\xi)^2/4(t-\tau)\right)}\exp{(-\lambda\xi)}d\xi=\\
\sqrt{\pi(t-\tau)}\exp{\left(-\lambda x+\lambda^2(t-\tau)\right)}\left(1+\text{erf}\left(\frac{x}{2\sqrt{t-\tau}}-\lambda\sqrt{t-\tau}\right)\right)\text{.}
\end{multline}
Therefore, we have:
\begin{equation}
\begin{split}
&\displaystyle\int_0^{+\infty}\exp{\left(-(x-\xi)^2/4(t-\tau)\right)}\sinh{(\lambda\xi)}d\xi=\\
&=\frac{\sqrt{\pi(t-\tau)}}{2}\exp{\left(\lambda^2(t-\tau)\right)}
\left(\exp{(\lambda x)}\left(1+\text{erf}\left(\frac{x}{2\sqrt{t-\tau}}+\lambda\sqrt{t-\tau}\right)\right)-\right.\\
&\left.\exp{(-\lambda x)}\left(1+\text{erf}\left(\frac{x}{2\sqrt{t-\tau}}-\lambda\sqrt{t-\tau}\right)\right)
\right)\text{.}
\end{split}
\end{equation}
By similar calculations, we have:
\begin{equation}
\begin{split}
&\displaystyle\int_0^{+\infty}\exp{\left(-(x+\xi)^2/4(t-\tau)\right)}\sinh{(\lambda\xi)}d\xi=\\
&=\frac{\sqrt{\pi(t-\tau)}}{2}\exp{\left(\lambda^2(t-\tau)\right)}
\left(\exp{(-\lambda x)}\left(1-\text{erf}\left(\frac{x}{2\sqrt{t-\tau}}-\lambda\sqrt{t-\tau}\right)\right)-\right.\\
&\left.\exp{(\lambda x)}\left(1-\text{erf}\left(\frac{x}{2\sqrt{t-\tau}}+\lambda\sqrt{t-\tau}\right)\right)
\right)\text{.}
\end{split}
\end{equation}
Then, we have:
\begin{equation}\label{app_2:Gphi2}
\displaystyle\int_0^{+\infty}G(x,t,\xi,\tau)\varphi_2(\xi)d\xi=\exp{\left(\lambda^2(t-\tau)\right)}\varphi_2(x)\text{,}\hspace{0.5cm}x>0\text{.}
\end{equation}

\item By the definitions of the functions $G$ and $\Phi=\varphi_3$ given in (\ref{th_Ir:G}) and (\ref{def:phi123}), respectively, we have:
\begin{equation}
\begin{split}
\displaystyle\int_0^{+\infty}G(x,t,\xi,\tau)\varphi_3(\xi)d\xi&=-\frac{\mu}{2\sqrt{\pi(t-\tau)}}\displaystyle\int_0^{+\infty}\left(\exp{\left(-(x-\xi)^2/4(t-\tau)\right)}\right.-\\
&\left.\exp{\left(-(x+\xi)^2/4(t-\tau)\right)}\right)\sin{(\lambda\xi)} d\xi\text{.}
\end{split}
\end{equation}

We first compute $\displaystyle\int_0^{+\infty}\exp{\left(-(x-\xi)^2/4(t-\tau)\right)}\sin{(\lambda\xi)}d\xi$.

\noindent By doing the change of variables $\zeta=(x-\xi)/2\sqrt{t-\tau}$, we have:
\begin{multline}
\displaystyle\int_0^{+\infty}\exp{\left(-(x-\xi)^2/4(t-\tau)\right)}\sin{(\lambda\xi)}d\xi=\\
2\sqrt{t-\tau}\displaystyle\int_{-\infty}^{x/2\sqrt{t-\tau}}\exp{(-\zeta^2)}\left(\sin{(\lambda x)}\cos{\left(2\lambda\sqrt{t-\tau}\zeta\right)}-\cos{(\lambda x)}\sin{\left(2\lambda\sqrt{t-\tau}\zeta\right)}\right)d\zeta\text{.}
\end{multline}
By using the identities (see \cite{NgGe1969}, p. 4):
\begin{equation}
\displaystyle\int\exp{(-\zeta^2)}\cos{(\alpha\zeta)}d\zeta=
\frac{\sqrt{\pi}}{4}\exp{\left(-\alpha^2/4\right)}
\left(\text{erf}\left(\zeta+\frac{\alpha}{2}i\right)+\text{erf}\left(\zeta-\frac{\alpha}{2}i\right)\right)
\end{equation}
and
\begin{equation}
\displaystyle\int\exp{(-\zeta^2)}\sin{(\alpha\zeta)}d\zeta=
\frac{\sqrt{\pi}i}{4}\exp{\left(-\alpha^2/4\right)}\left(\text{erf}\left(\zeta+\frac{\alpha}{2}i\right)-\text{erf}\left(\zeta-\frac{\alpha}{2}i\right)\right)\text{,}
\end{equation}
where $\alpha\in\mathbb{R}$ and $i$ denotes the imaginary unit, we have:
\begin{multline}
\displaystyle\int_{-\infty}^{x/2\sqrt{t-\tau}}\exp{(-\zeta^2)}\cos{\left(2\lambda\sqrt{t-\tau}\zeta\right)}d\zeta=\\
\frac{\sqrt{\pi}}{4}\exp{\left(-\lambda^2(t-\tau)\right)}\left(\text{erf}\left(\frac{x}{2\sqrt{t-\tau}}+i\lambda\sqrt{t-\tau}\right)+
\text{erf}\left(\frac{x}{2\sqrt{t-\tau}}-i\lambda\sqrt{t-\tau}\right)+2\right)
\end{multline}
and
\begin{multline}
\displaystyle\int_{-\infty}^{x/2\sqrt{t-\tau}}\exp{(-\zeta^2)}\sin{\left(2\lambda\sqrt{t-\tau}\zeta\right)}d\zeta=\\
\frac{\sqrt{\pi}i}{4}\exp{\left(-\lambda^2(t-\tau)\right)}\left(\text{erf}\left(\frac{x}{2\sqrt{t-\tau}}+i\lambda\sqrt{t-\tau}\right)-
\text{erf}\left(\frac{x}{2\sqrt{t-\tau}}-i\lambda\sqrt{t-\tau}\right)\right)\text{.}
\end{multline}
Then, we have:
\begin{equation}
\begin{split}
&\displaystyle\int_0^{+\infty}\exp{\left(-(x-\xi)^2/4(t-\tau)\right)}\sin{(\lambda\xi)}d\xi=\\
&\frac{\sqrt{\pi(t-\tau)}}{2}\exp{(-\lambda^2(t-\tau))}
\left(\text{erf}\left(\frac{x}{2\sqrt{t-\tau}}+i\lambda\sqrt{t-\tau}\right)\sin{(\lambda x)}+\right.\\
&\left.\text{erf}\left(\frac{x}{2\sqrt{t-\tau}}-i\lambda\sqrt{t-\tau}\right)\sin{(\lambda x)}-\text{erf}\left(\frac{x}{2\sqrt{t-\tau}}+i\lambda\sqrt{t-\tau}\right)i\cos{(\lambda x)}\right.+\\
&\left.\text{erf}\left(\frac{x}{2\sqrt{t-\tau}}-i\lambda\sqrt{t-\tau}\right)i\cos{(\lambda x)}+2\sin{(\lambda x)}\right)\text{.}
\end{split}
\end{equation}
By similar calculations, we have:
\begin{equation}
\begin{split}
&\displaystyle\int_0^{+\infty}\exp{\left(-(x+\xi)^2/4(t-\tau)\right)}\sin{(\lambda\xi)}d\xi=\\
&\frac{\sqrt{\pi(t-\tau)}}{2}\exp{(-\lambda^2(t-\tau))}
\left(\text{erf}\left(\frac{x}{2\sqrt{t-\tau}}-i\lambda\sqrt{t-\tau}\right)i\cos{(\lambda x)}-\right.\\
&\left.\text{erf}\left(\frac{x}{2\sqrt{t-\tau}}+i\lambda\sqrt{t-\tau}\right)i\cos{(\lambda x)}+\text{erf}\left(\frac{x}{2\sqrt{t-\tau}}+i\lambda\sqrt{t-\tau}\right)\sin{(\lambda x)}\right.+\\
&\left.\text{erf}\left(\frac{x}{2\sqrt{t-\tau}}-i\lambda\sqrt{t-\tau}\right)\sin{(\lambda x)}-2\sin{(\lambda x)}\right)\text{.}
\end{split}
\end{equation}
Then, we have:
\begin{equation}\label{app_2:Gphi3}
\displaystyle\int_0^{+\infty}G(x,t,\xi,\tau)\varphi_3(\xi)d\xi=\exp{\left(-\lambda^2(t-\tau)\right)}\varphi_3(x)\text{,}\hspace{0.5cm}x>0\text{.}
\end{equation}
\end{enumerate}
The proofs of propositions \ref{pr:Iru1}, \ref{pr:Iru2} and  \ref{pr:Iru3} follow from the expression for $u$ given in (\ref{th_Ir:u}), the expression for $\displaystyle\int_0^{+\infty}G(x,t,\xi,0)h(\xi)d\xi$ obtained in (\ref{app_2:Gh}) and the expression for $\displaystyle\int_0^{+\infty}G(x,t,\xi,\tau)\Phi(\xi)d\xi$ obtained in (\ref{app_2:Gphi1}), (\ref{app_2:Gphi2}) and (\ref{app_2:Gphi3}), respectively.

\end{document}